\documentclass[12pt]{article}                                                                                                                                                                    

\usepackage{amsthm,amsmath,amssymb,graphics,graphicx,hyperref,epstopdf,enumerate}
\usepackage{tikz}
\usetikzlibrary{shapes}
\usepackage{multicol}
\title{The Dimension of the Negative Cycle Vectors\\ of Signed Graphs}
\author{Alex Schaefer\footnote{Dept.\ of Mathematical Sciences, Binghamton University, Binghamton, NY 13902-6000, U.S.A.; aschaef3@binghamton.edu} and Thomas Zaslavsky\footnote{Dept.\ of Mathematical Sciences, Binghamton University, Binghamton, NY 13902-6000, U.S.A.; zaslav@math.binghamton.edu}}
\date{\today}

\newtheorem{theorem}{Theorem}[section]

\newtheorem{corollary}[theorem]{Corollary}
\newtheorem{lemma}[theorem]{Lemma}
\theoremstyle{definition}

\newtheorem{conjecture}[theorem]{Conjecture}

\numberwithin{equation}{section}

\newcommand\bbR{\mathbb{R}}
\newcommand{\G}{\Gamma}  
\renewcommand{\S}{\Sigma}  
\newcommand{\s}{\sigma}

\newcommand\NCV{\operatorname{NCV}}
\DeclareMathOperator{\Spec}{SpecC}

\newcommand\ssum{\operatorname*{\sum\sum}}
\newcommand\odd{\mathrm{odd}}
\newcommand\even{\mathrm{even}}

\newcommand{\ssection}[1]{%
  \section[#1]{\centering\normalfont\scshape #1}}
\newcommand{\ssubsection}[1]{%
  \subsection[#1]{\raggedright\normalfont\itshape #1}}

\begin{document} 

\maketitle

\tableofcontents

\begin{abstract}
A \emph{signed graph} is a graph $\G$ where the edges are assigned sign labels, either ``$+$'' or ``$-$''. The sign of a cycle is the product of the signs of its edges. Let $\Spec(\G)$ denote the list of lengths of cycles in $\G$. We equip each signed graph with a vector whose entries are the numbers of negative $k$-cycles for $k\in\Spec(\G)$. These vectors generate a subspace of $\bbR^{\Spec(\G)}$. Using matchings with a strong permutability property, we provide lower bounds on the dimension of this space; in particular, we show for complete graphs, complete bipartite graphs, and a few other graphs that this space is all of $\bbR^{\Spec(\G)}$.
\end{abstract}

\ssection{Introduction}

A \emph{signed graph} $\S$ is a graph $\G$ whose edges have sign labels, either ``$+$'' or ``$-$''. The sign of a cycle in the graph is the product of the signs of its edges.  Write $c_l^-(\S)$ for the number of negative cycles of length $l$ in $\S$ and collect these numbers in the \emph{negative cycle vector} $c^-(\S) = (c_3^-,c_4^-,\ldots,c_n^-)\in\bbR^{n-2}$, where $n$ is the order of $\S$.  We are interested in the structure of the collection $\NCV(\G)$ of all negative cycle vectors of signings of a fixed underlying simple graph $\G$.  

There are (at least) three natural questions raised by the existence of these collections of vectors.  
Most simply, what is their dimension? This is the question we address here.  
The \emph{cycle spectrum} $\Spec(\G)$ is the list of lengths of cycles in $\G$; $\NCV(\G)$ is a subset of $\bbR^{\Spec(\G)}$ and generates an affine subspace (which is a linear subspace since the negative cycle vector $c^-(+\G)$ corresponding to the all-positive signing is the zero vector).  
We develop a general approach to the dimension question in terms of ``permutable matchings'' (see Section \ref{pmatch}) that allows us to prove for $\G = K_n$, $K_{m,n}$, and the Petersen graph that $\NCV(\G)$ has dimension $|\Spec(\G)|$; it also gives us a lower bound for the Heawood graph and one other graph family.  
(We also solve a few examples with an \emph{ad hoc} method.)

Secondly, what is their convex hull? In \cite{PopTom1996} and \cite{Tomescu1976}, Popescu and Tomescu gave inequalities bounding the numbers of negative cycles in a signed complete graph, which is a step towards the answer for $K_{n}$.  
A related question: Do the facets of the convex cone generated by $\NCV(\G)$ have combinatorial meaning?

Finally, which vectors in the convex hull are actually the vectors of signed graphs? 
Recently Kittipassorn and M\'esz\'aros \cite{KittipassornMeszaros} gave strong restrictions on the number of negative triangles in a signed $K_{n}$.  Again, this provides a step towards that answer.

Our work was originally motivated by the complete graph and a natural extension to complete bipartite graphs.  Those cases and others led to the following plausible conjecture.

\begin{conjecture}[Schaefer, 2017]\label{Conj:S}
For any graph $\G$, $\dim\NCV(\G) = |\Spec(\G)|$.
\end{conjecture}

\ssection{Background}

\ssubsection{Graphs}\label{g}

A \emph{graph} is a pair $\G=(V,E)$, where $V=\{v_{1},\dots,v_{n}\}$ is a (finite) set of \emph{vertices} and $E$ is a (finite) set of unordered pairs of vertices, called \emph{edges}. Our graphs are all unlabeled, simple, and undirected.  Thus, all cycle lengths are between 3 and $n$.

The number of cycles of length $l$ in $\G$ is $c_l=c_l(\G)$.  The cycle vector of $\G$ is $c(\G)=(c_3,c_4,\ldots,c_n)$; sometimes we omit the components that correspond to lengths $l$ not in the cycle spectrum.

\ssubsection{Signed graphs}\label{sg}

A \emph{signed graph} is a triple $\S=(V,E,\s)$ where $\G=(V,E)$ is a graph (the \emph{underlying graph} of $\S$) and $\s:E\to\{+,-\}$ is the \emph{sign function}.
The \emph{sign of a cycle} is the product of the signs of its edges; a signed graph in which every cycle is positive is called \emph{balanced}.  
The \emph{negative edge set} $E^-$ is the set of negative edges of $\S$ and the \emph{negative subgraph} is $\S^-=(V,E^-)$, the spanning subgraph of negative edges.  
We sometimes write $\G_N$ for $\G$ signed so that $N$ is its set of negative edges.  

\emph{Switching} $\S$ means choosing a vertex subset $X \subseteq V$ and negating all the edges between $X$ and its complement.  
Switching yields an equivalence relation on the set of all signings of a fixed underlying graph.
If $\S_{2}$ is isomorphic to a switching of $\S_{1}$, we say that $\S_{1}$ and $\S_{2}$ are \emph{switching isomorphic}.  This relation is an equivalence relation on signed graphs; we denote the equivalence class of $\S$ by $[\S]$.
A signed graph is balanced if and only if it is switching isomorphic to the all-positive graph.
Signed graphs that are switching isomorphic (like those in Figure \ref{K6equiv}) have the same negative cycle vector.

As with $c(\G)$, we may omit the components of $c^-(\S)$ that correspond to lengths $l$ not in the cycle spectrum.  
Also, we may write either $c^-(\S)$ or $c^-(\s)$, the latter when only the signature $\s$ is varying.


\begin{figure}[h]
\begin{multicols}{2}
\begin{center}
\begin{tikzpicture}[scale=.2]

\node[fill, shape=circle] (1) at (210:8) {};
\node[fill, shape=circle] (2) at (330:8) {};
\node[fill, shape=circle] (3) at (150:8) {};
\node[fill, shape=circle] (4) at (30:8) {};
\node[fill, shape=circle] (5) at (90:8) {};
\node[fill, shape=circle] (6) at (270:8) {};

\draw[line width=2pt, style=loosely dashed] (1)--(2) ;
\draw[line width=2pt] (1)--(3) ;
\draw[line width=2pt] (1)--(4) ;
\draw[line width=2pt] (1)--(5) ;
\draw[line width=2pt] (1)--(6) ;
\draw[line width=2pt] (2)--(3) ;
\draw[line width=2pt] (2)--(4) ;
\draw[line width=2pt] (2)--(5) ;
\draw[line width=2pt] (2)--(6) ;
\draw[line width=2pt] (3)--(4) ;
\draw[line width=2pt, style=loosely dashed] (3)--(5) ;
\draw[line width=2pt] (3)--(6) ;
\draw[line width=2pt, style=loosely dashed] (4)--(5) ;
\draw[line width=2pt] (4)--(6) ;
\draw[line width=2pt, style=loosely dashed] (5)--(6) ;

\end{tikzpicture}\\
\columnbreak
\begin{tikzpicture}[scale=.2]

\node[fill, shape=circle] (1) at (210:8) {};
\node[fill, shape=circle] (2) at (330:8) {};
\node[fill, shape=circle] (3) at (150:8) {};
\node[fill, shape=circle] (4) at (30:8) {};
\node[fill, shape=circle] (5) at (90:8) {};
\node[fill, shape=circle] (6) at (270:8) {};

\draw[line width=2pt, style=loosely dashed] (1)--(2) ;
\draw[line width=2pt] (1)--(3) ;
\draw[line width=2pt] (1)--(4) ;
\draw[line width=2pt, style=loosely dashed] (1)--(5) ;
\draw[line width=2pt] (1)--(6) ;
\draw[line width=2pt] (2)--(3) ;
\draw[line width=2pt] (2)--(4) ;
\draw[line width=2pt, style=loosely dashed] (2)--(5) ;
\draw[line width=2pt] (2)--(6) ;
\draw[line width=2pt] (3)--(4) ;
\draw[line width=2pt] (3)--(5) ;
\draw[line width=2pt] (3)--(6) ;
\draw[line width=2pt] (4)--(5) ;
\draw[line width=2pt] (4)--(6) ;
\draw[line width=2pt] (5)--(6) ;

\end{tikzpicture}\\
\end{center}
\end{multicols}
\caption{Two switching equivalent signings of $K_{6}$, with the same negative cycle vector $(10,18,36,36)$.  Solid lines are positive, dashed lines are negative.}
\label{K6equiv}
\end{figure}
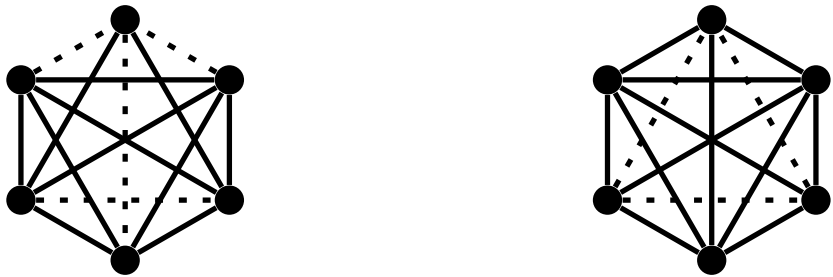

The \emph{negation} of $\S$ is $-\S=(V,E,-\s)$, in which the sign of every edge is negated.  Sometimes $\S$ and $-\S$ are switching isomorphic, e.g., when $\S$ is bipartite or when it is a signed complete graph whose negative subgraph is self-complementary.

\ssubsection{Permutable matchings}\label{pmatch}

A \emph{matching} in $\G$ is a set $M$ of pairwise nonadjacent edges; it is \emph{perfect} if $V(M)=V$.  
A matching $M$ (or any other edge set) is \emph{permutable} if the automorphism group of $\G$ acts on the edges of $M$ as the symmetric group $S_{|M|}$.
We base our results largely on permutable matchings, after Zaslavsky noticed their utility in proving our results for complete and complete bipartite graphs.
The advantage of permutability is that, in counting negative cycles using a permutable matching, any two equicardinal subsets belong to the same number of negative cycles of each length.  That makes it feasible to calculate the numbers in the vectors we use to estimate the dimension of $\NCV(\G)$.  

Our introduction of permutable matchings led to the question: Which graphs have permutable matchings?  That has been investigated by Schaefer and Swartz in \cite{SS2017}; they find large families of examples.  On the other hand, there are only a few kinds of graph with permutable perfect matchings; Schaefer and Swartz determine them all.

\ssection{Rank and Dimension}\label{rankdim}

The dimension of $\NCV(\G)$ is the rank of the matrix whose rows are the negative cycle vectors of all signatures of $\G$.  (The columns of this matrix may be regarded as corresponding to all lengths $k \in \{3,4,\ldots,n\}$, or only the lengths in $\Spec(\G)$, depending on which is more convenient.  The column of $k \notin \Spec(\G)$, if included, is all zero.)  We know the rank  cannot be greater than $|\Spec(\G)|$, the number of nonzero columns, so if we produce a submatrix of that rank we have proved that $\dim\NCV(\G) = |\Spec(\G)|$.  That is what we now endeavor to do with the aid of a permutable matching.  

Even if permutable matchings fail to reach the spectral upper bound, they imply a lower bound.  
However, we are happy to say that in our three main examples, permutable matchings solve the dimension problem.

The rank of a matrix $A$ is written $\mathrm{rk}(A)$.

\ssubsection{Any negative edge set}

We begin with the most general calculation.  
Given a signed graph $\G_N$ with an arbitrary negative edge set $N \subseteq E$, how many negative cycles are there of each length?  For $X \subseteq N$ let $f_l(X) := $ the number of $l$-cycles that intersect $N$ precisely in $X$.  
We get a formula for $f$ by M\"obius inversion from $g_l(X) :=$ the number of $l$-cycles that contain $X$, since
$$
g_l(X) = \sum_{X \subseteq Y \subseteq N} f_l(Y),
$$ 
which implies that
$$
f_l(X) = \sum_{X \subseteq Y \subseteq N} (-1)^{|Y|-|X|} g_l(Y).
$$
The number of negative $l$-cycles is the number of $l$-cycles that intersect $N$ in an odd number of edges; therefore,
\begin{align}
c_l^-(\G_N) = \sum_{X \subseteq N,\, |X|\text{ odd}} f_l(X) 
&= \ssum_{X \subseteq Y \subseteq N,\, |X|\text{ odd}} (-1)^{|Y|-|X|} g_l(Y) \notag\\
&= \sum_{Y \subseteq N} g_l(Y) \sum_{X \subseteq Y,\, |X|\,\text{odd}} (-1)^{|Y|-|X|} \notag\\
&= \sum_{\emptyset \neq Y \subseteq N} (-2)^{|Y|-1} g_l(Y).
\label{E:c_l}
\end{align}
This applies to every underlying graph $\G$.

\ssubsection{A matrix calculation}

Now assume we have a graph $\G$ of order $n$ and $m$ unbalanced sign functions $\s_1,\ldots,\s_m$ in addition to the all-positive function $\s_0\equiv+$.  To avoid redundancy we want the associated signed graphs to be switching nonisomorphic.  For instance, choosing more than half the edges at a vertex to be negative is switching equivalent to choosing fewer than half, so we would not want the negative edge set to contain more than $\frac12(\deg(v)-1)$ of the edges incident with any vertex $v$.

For the present assume $n$ is even.   
Here is the matrix of the negative cycle vectors of all signings $\s_s$ and their negatives, with columns segregated by parity.  
The rows are one for $+\G$ ($\s_0\equiv+$), then $m$ rows for the unbalanced signatures $\s_s$, $0<s\leq m$, then $-\G$ (the signature $-\s_0\equiv-$), then the $m$ negations $-\s_s$.  
The relationship between the upper and lower halves is that 
\[
c_l^-(-\s_s) = \begin{cases}
c_l - c_l^-(\s_s) &\text{ if $l$ is odd.} \\
c_l^-(\s_s) &\text{ if $l$ is even.}
\end{cases}
\]
The resulting matrix is
\begin{align}
{\small
\left(
\begin{array}{cccc|cccc}
0	&0	&\cdots	&0		&0	&0	&\cdots	&0	\\
c_3^-(\s_1)	&c_5^-(\s_1)	&\cdots	&c_{n-1}^-(\s_1)	&c_4^-(\s_1)	&c_6^-(\s_1)	&\cdots	&c_n^-(\s_1)	\\
\vdots	&\vdots	&\ddots	&\vdots		&\vdots	&\vdots	&\cdots	&\vdots	\\
c_3^-(\s_m)	&c_5^-(\s_m)	&\cdots	&c_{n-1}^-(\s_m)	&c_4^-(\s_m)	&c_6^-(\s_m)	&\cdots	&c_n^-(\s_m)	\\
c_3	&c_5	&\cdots	&c_{n-1}		&0	&0	&\cdots	&0	\\
c_3-c_3^-(\s_1) &c_5-c_5^-(\s_1) &\cdots &c_{n-1}-c_{n-1}^-(\s_1) &c_4^-(\s_1) &c_6^-(\s_1) &\cdots &c_n^-(\s_1) \\
\vdots	&\vdots	&\ddots	&\vdots	&\vdots	&\vdots	&\cdots	&\vdots	\\
c_3-c_3^-(\s_m) &c_5-c_5^-(\s_m) &\cdots&c_{n-1}-c_{n-1}^-(\s_m)	&c_4^-(\s_m) &c_6^-(\s_m) &\cdots &c_n^-(\s_m) \\
\end{array}
\right).
}
\label{E:neg cycle matrix}
\end{align}
Row operations reduce this matrix to 
\begin{align}
&\quad
{
\left(
\begin{array}{cccc|cccc}
0	&0	&\cdots	&0		&0	&0	&\cdots	&0	\\
c_3^-(\s_1)	&c_5^-(\s_1)	&\cdots	&c_{n-1}^-(\s_1)	&0	&0	&\cdots	&0	\\
\vdots	&\vdots	&\ddots	&\vdots		&\vdots	&\vdots	&\cdots	&\vdots	\\
c_3^-(\s_m)	&c_5^-(\s_m)	&\cdots	&c_{n-1}^-(\s_m)	&0	&0	&\cdots	&0	\\
c_3	&c_5	&\cdots	&c_{n-1}		&0	&0	&\cdots	&0	\\
0	&0	&\cdots	&0		&c_4^-(\s_1)	&c_6^-(\s_1)	&\cdots	&c_n^-(\s_1) \\
\vdots	&\vdots	&\ddots	&\vdots		&\vdots	&\vdots	&\cdots	&\vdots	\\
0	&0	&\cdots	&0		&c_4^-(\s_m)	&c_6^-(\s_m)	&\cdots	&c_n^-(\s_m) \\
\end{array}
\right).
}
\label{E:reduced nc matrix even}
\end{align}
Ignoring the first row of zeroes, this is a block matrix 
$$
A:=\begin{pmatrix} U & O \\ c_{\mathrm{odd}}(\G) & \mathbf0 \\[2pt] O & R \end{pmatrix}.
$$
The middle row $c_{\text{odd}}(\G)$, consisting of the odd-cycle numbers of $\G$, corresponds to $-\G$.  The upper left block $U$ is the matrix of negative odd-cycle vectors of the unbalanced signatures $\s_s$, and the lower right block $R$ is the matrix of negative even-cycle vectors of the same signatures.  We infer the fundamental fact that:

\begin{lemma}\label{T:neg cycle matrix}
The rank of the negative cycle matrix \eqref{E:neg cycle matrix} 
equals the sum of the ranks of $\begin{pmatrix} U \\ c_{\mathrm{odd}}(\G) \end{pmatrix}$ and $R$.
\end{lemma}

Lemma \ref{T:neg cycle matrix} is written for even $n$ but by putting into $U$ a column for $c^-_{n+1}$ we include the odd cycles of order $n+1$ ($n$ still being even).  This can be handled by the same computation.  The reduced matrix in this case is
\begin{align}
{
\left(
\begin{array}{ccccc|cccc}
0	&0	&\cdots	&0	&0		&0	&0	&\cdots	&0	\\
c_3^-(\s_1)	&c_5^-(\s_1)	&\cdots	&c_{n-1}^-(\s_1)&c_{n+1}^-(\s_1)	&0	&0	&\cdots	&0	\\
\vdots	&\vdots	&\ddots	&\vdots	&\vdots		&\vdots	&\vdots	&\cdots	&\vdots	\\
c_3^-(\s_m)	&c_5^-(\s_m)	&\cdots	&c_{n-1}^-(\s_m)&c_{n+1}^-(\s_m)	&0	&0	&\cdots	&0	\\
c_3	&c_5	&\cdots	&c_{n-1}	&c_{n+1}		&0	&0	&\cdots	&0	\\
0	&0	&\cdots	&0	&0		&c_4^-(\s_1)	&c_6^-(\s_1)	&\cdots	&c_n^-(\s_1) \\
\vdots	&\vdots	&\ddots	&\vdots	&\vdots		&\vdots	&\vdots	&\cdots	&\vdots	\\
0	&0	&\cdots	&0	&0		&c_4^-(\s_m)	&c_6^-(\s_m)	&\cdots	&c_n^-(\s_m) \\
\end{array}
\right).
}
\label{E:reduced nc matrix odd}
\end{align}

For a bipartite graph $U=O$ and $c_{\mathrm{odd}}=\mathbf0$, so only $R$ needs to be considered.

\ssubsection{Permutable negative matchings}\label{pnegm}

Henceforth we assume we have chosen a fixed permutable matching $M_m$ of $m$ edges in $\G$.  
For each $s=1,2,\ldots,m$ we choose a submatching $M_s \subseteq M_m$ of $s$ edges and we define the signature $\s_s$ as that of the signed graph $\G_{M_s}$.  (It does not matter which $M_s$ we use, because $M_m$ is permutable.)  This generates a matrix of negative cycle vectors as in \eqref{E:neg cycle matrix}.  

In particular, in $K_n$ the biggest permutable edge set is a perfect or near-perfect matching.  This turns out to be ``perfect'' for our purposes.  (An almost equally big set is half the edges incident to one vertex, but we found that to be useless since then the entire matrix \eqref{E:neg cycle matrix} has rank 1.)

Permutability implies that $g_l(Y)$ depends only on $|Y|$ so we may define $G_l(k)=g_l(Y)$ for any one $k$-edge subset $Y\subseteq M_m$.  
Then \eqref{E:c_l} becomes
\begin{align}
c_l^-(\G_{M_s}) &= \sum_{k=1}^{s} (-2)^{k-1} \binom{s}{k} G_l(k) 
= \sum_{k=1}^{n} (-2)^{k-1} \frac{G_l(k)}{k!}  (s)_k,
\label{E:Ms}
\end{align}
where $(x)_k$ denotes the falling factorial, $(x)_k=x(x-1)\cdots(x-[k-1])$.  
Formula \eqref{E:Ms} gives $c_l^-(\G_{M_s})$ as a polynomial function $p_l(s)$ without constant term, of degree $d_l$ where $d_l$ as the largest integer $k$ for which $G_l(k) > 0$; that is, $d_l$ is the largest size of a submatching of $M_m$ that is contained in some cycle of length $l$.  (We leave $d_l$ undefined if no $l$-cycle intersects $M_m$.)  Clearly, $d_l\leq m$.

(Our reasoning works equally well for subsets of any permutable edge set $N$ in any graph.  It is easy to see that there are only three possible kinds of permutable set: a matching, a subset of the edges incident to a vertex, and the three edges of a triangle.  We mentioned that a permutable set of edges at a vertex is useless for $K_n$.  We have not seen a graph where a triangle's edges might help find the dimension.)

We illustrate our calculations with $K_n$ as a running example.  The data is from Section \ref{Kn}.  Let $m=\lfloor n/2 \rfloor$.  The number of $l$-cycles in $K_n$ that intersect a maximum matching $M_m$ in a fixed set of $k$ edges is
\[
G_{l}(k)=\binom{n-2k}{l-2k}(l-k-1)!\cdot2^{k-1}.
\] 

A column of $U$ or $R$ is not all zero if and only if it corresponds to a cycle length $l$ for which there exists an $l$-cycle in $\G$ that intersects $M_m$.  Such a column contains $m$ values of the polynomial $p_l(s)$.  
Since $p_l$ has degree at most $m$ and no constant term, these values determine $p_l$ completely.

Now a nonzero column in $U$ or $R$ for cycle length $l$ looks like this:
\begin{equation}
\begin{pmatrix}
p_l(1) \\ p_l(2) \\ \vdots \\ p_l(m)
\end{pmatrix}
=
\begin{pmatrix}
\alpha_l 1^{d_l}+\cdots \\ \alpha_l 2^{d_l}+\cdots \\ \vdots \\ \alpha_l m^{d_l}+\cdots
\end{pmatrix},
\label{E:column}
\end{equation}
since $p_l$ is a polynomial of degree $d_l$; here $\alpha_l = (-2)^{d_l-1} {G_l(d_l)}/{d_l!}$.

Suppose the set $\Delta_\odd:=\{d_3,d_5,\ldots,d_{n-1}\}$ has $\delta_\odd$ (distinct) elements and the set $\Delta_\even:=\{d_4,d_6,\ldots,d_n\}$ has $\delta_\even$ elements.  The number of polynomial degrees represented in the columns of $U$ is $\delta_\odd$ (which may be less than the number of nonzero columns), and similarly for $R$.

In $K_n$ with a maximum matching, $\Delta_\odd=\{3,5,\ldots\}$ (odd numbers up to $n$) and $\Delta_\even=\{4,6,\ldots\}$ (even numbers up to $n$).

\begin{lemma}\label{L:UR ranks}
The rank of $U$ is at least $\delta_\odd$ and that of $R$ is at least $\delta_\even$.  

The rank of $\begin{pmatrix} U \\ c_\odd \end{pmatrix}$ is $\mathrm{rk}(U)+1$ if there is an odd length $l$ such that an $l$-cycle exists in $\G$ but no $l$-cycle intersects $M_m$.
\end{lemma}

\begin{proof}
In $U$ choose one column of each different degree $d_l$.  Divide by the leading coefficient $\alpha_l$; this does not affect the rank.  Now add columns of the form $\begin{pmatrix} l^d \end{pmatrix}_{s=1}^{m}$ for every $d=1,2,\ldots,m$ that is not in $\Delta_\odd$.  Column operations allow us to eliminate the lower-degree terms of the column \eqref{E:column}, leaving a Vandermonde matrix with $1^d$ in the top row and $m^d$ in the bottom row of column $d$ for each $d=1,2,\ldots,m$.  The rank of is $m$.  Now reverse the column operations; the rank remains the same, so the columns of $U$ must have full column rank.

The same reasoning applies to $R$.

The extra 1 in the rank of $\begin{pmatrix} U \\ c_\odd \end{pmatrix}$ arises from the fact that, under the assumption, it has a column that is zero in $U$ but is nonzero in $c_\odd$.
\end{proof}

By this lemma, for $K_n$ the ranks of $U$ and $R$ are $\lceil n/2 \rceil -1$ and $\lfloor n/2 \rfloor -1$, respectively, which sum to $n-2$.

\ssubsection{Theorems}

Lemma \ref{L:UR ranks} yields our principal general theorem.  Given a matching $M_m$ and a cycle length $l \in \Spec(\G)$, define
\[
\mu(l) := \max_{C_l} |C_l \cap M_m|,
\]
maximized over all $l$-cycles $C_l$.

\begin{theorem}\label{T:main}
Let $M_m$ be a permutable $m$-matching in  $\G$.  Then 
\begin{gather*}
|\{ \mu(l) : \text{odd } l \in \Spec(\G) \}| + |\{ \mu(l)>0 : \text{even } l \in \Spec(\G) \}| \\
\leq \dim\NCV(\G) \leq |\Spec(\G)|.
\end{gather*}

Suppose that every even cycle length, and all odd cycle lengths with at most one exception, are values of $\mu(l)$.  
Then $\NCV(\G)$ spans $\bbR^{|\Spec(\G)|}$.
\end{theorem}

\begin{proof}
The value of $\mu(l)$ is the degree $d_l$ of the polynomials $p_l(s)$ if such a polynomial exists.  
The polynomial exists and $d_l$ is defined if and only if some $C_l \cap M_m \neq \emptyset$, in other words if and only if $\mu(l)>0$.  Thus, there is a value $\mu(l)=0$ for some odd $l \in \Spec(\G)$ if and only if $\mathrm{rk}\begin{pmatrix} U \\ c_\odd \end{pmatrix} = \mathrm{rk}(U)+1$.
\end{proof}

There is a simpler statement that applies to graphs with a sufficiently omnipresent permutable matching.  Given $m$, define $\nu_\odd(m) :=$ the number of odd lengths $l< 2m$ in $\Spec(\G)$, $+1$ if there is an odd cycle length $l\geq2m$, and define $\nu_\even(m) :=$ the number of even lengths $l< 2m$ in $\Spec(\G)$, $+1$ if there is an even cycle length $l\geq2m$.

\begin{theorem}\label{T:omni}
Suppose $M_m$ is a permutable $m$-matching in $\G$ and for every length $l \in \Spec(\G)$ there exists a cycle $C_l$ such that $|C_l\cap M_m| = \min( m, \lfloor l/2 \rfloor$).
Then $\dim\NCV(\G) \geq \nu_\odd(m)+\nu_\even(m)$.
\end{theorem}

The hypothesis can be lessened since, if there is any cycle length $l\geq2m$, it suffices to have one length $l\geq 2m$ for which there is a $C_l$ with $|C_l\cap M_m| = m$.

\begin{proof}
The hypotheses imply that
\[
d_l = \begin{cases}
\lfloor l/2 \rfloor &\text{ if } l \leq 2m, \\
m &\text{ if } l \geq 2m.
\end{cases}
\]
We count the number of distinct values $d_l$ for odd and even cycle lengths.  For odd $l$ we get $(l-1)/2$ if $l\in\Spec(\G)$ and $l<2m$, and we get $m$ if and only if there exists a cycle length $l\geq 2m$.  The total is $\nu_\odd$.  The computation of $\nu_\even$ is similar.

The values of $\mu(l)$ in Theorem \ref{T:main} are the same as those of $d_l$ unless there is a cycle length for which no $l$-cycle intersects $M_m$; but that is ruled out by our hypotheses.  Theorem \ref{T:omni} follows.
\end{proof}

A graph is \emph{bipancyclic} if it is bipartite and has a cycle of every even length from 4 to $n$.

\begin{corollary}\label{C:pan}
Assume $\G$ is pancyclic and has a permutable $m$-matching $M_m$, and for every $l$ with $3\leq l\leq n$ there is an $l$-cycle $C_l$ with $|C_l\cap M_m| = \min( m, \lfloor l/2 \rfloor$).  Then $\dim\NCV(\G) = n-2$ if $2m \geq n-1$, and $n-2\geq\dim \NCV(\G) \geq 2m-1$ if $2m\leq n-2$.

Assume $\G$ is bipancyclic and has vertex class sizes $p, q$ with $p\leq q$, and it has a permutable $m$-matching $M_m$ such that for every $k$ with $2\leq k\leq p$ there is a $2k$-cycle $C_{2k}$ with $|C_{2k} \cap M_m| = \min( m, k)$.  
Then $\dim\NCV(\G) = p-1$ if $m = p$, and $p-1\geq\dim \NCV(\G) \geq m-1$ if $m\leq p-1$.
\end{corollary}

The hypotheses can be lessened in the same way as those of Theorem \ref{T:omni}.

\begin{proof}
If $\G$ is pancyclic, $\nu_\odd$ counts all the numbers $3,5,\ldots,2m-1$ plus 1 for $2m+1$ if $n>2m$, and $\nu_\even$ counts the numbers $4,6,\ldots,2m-2$ plus 1 for $2m$ since $n\geq 2m$.  Thus 
\[
\nu_\odd+\nu_\even = \begin{cases}
(m)+(m-1) = 2m-1 &\text{ if } n > 2m, \\
(m-1)+(m-1) = 2m-2 &\text{ if } n=2m.
\end{cases}
\]
The conclusion follows easily.

If $\G$ is bipancyclic, then $\nu_\even=m-1$ and the conclusion follows easily.
\end{proof}

The two most complete graphs are easy consequences of any of the preceding results, but especially of Corollary \ref{C:pan}.

\begin{corollary}\label{C:Kn Kpq}
For a complete graph $K_n$ with $n\geq 3$, $\dim \NCV(K_n) = n-2.$

For a complete bipartite graph $K_{p,q}$ with $p,q \geq 2$, $\dim \NCV(K_{p,q}) = \min(p,q)-1.$
\end{corollary}

\ssection{Examples}

\ssubsection{The Compleat Complete Graph}\label{Kn}

We need to supply a missing computation for $K_{n}$.  But first, let us see the negative cycle vectors of all signings of small complete graphs.

The vectors for $K_{3}$ are
\[
(0), \ (1)
\]
(from the balanced and unbalanced triangle).  The vectors for $K_{4}$ are
\[
(0,0), \ (2,2), \ (4,0)
\]
(the all-positive graph, one negative edge, and two nonadjacent negative edges). Here are the vectors for $K_{5}$:
\[
(0,0,0), \ (3,6,6), \ (4,8,8), \ (5,10,6), \ (6,8,4), \ (7,6,6), \ (10,0,12);
\]
and for $K_{6}$:
\[
\begin{array}{llll}
(0,0,0,0), &(4,12,24,24), & (6,18,36,36), &(8,20,32,24), \\
(10,18,36,36), &(8,24,40,32), & (10,22,36,28), &(12,24,24,32), \\
(10,26,36,28), &(8,24,48,32), & (14,18,36,36), &(12,24,32,32), \\
(12,20,40,24), &(10,30,36,20), & (16,12,48,24), &(20,0,72,0). 
\end{array}
\]

The number of switching isomorphism classes of complete graphs grows super-exponentially \cite{MallowsSloane}.  Since two signed graphs which yield different vectors must belong to different classes, one naturally wonders about the converse property, that the vector uniquely identifies a switching class.  This is true up through $K_{7}$ but false for $K_{8}$: see Figure \ref{uniquenesscounterex} below (found by Gary Greaves, whose assistance we greatly appreciate).  Thus when $n\geq8$ there are (certainly when $n=8$ and surely also for all larger orders) fewer vectors than classes, but in general there will still be a very large number.

\begin{figure}[ht!!!]\label{uniquenesscounterex}
\begin{multicols}{2}
\begin{center}
\begin{tikzpicture}[scale=.3]

\node[fill, shape=circle] (1) at (247.5:7) {};
\node[fill, shape=circle] (2) at (292.5:7) {};
\node[fill, shape=circle] (3) at (202.5:7) {};
\node[fill, shape=circle] (4) at (337.5:7) {};
\node[fill, shape=circle] (5) at (157.5:7) {};
\node[fill, shape=circle] (6) at (22.5:7) {};
\node[fill, shape=circle] (7) at (112.5:7) {};
\node[fill, shape=circle] (8) at (67.5:7) {};

\draw[line width=2pt, style=loosely dashed] (1)--(2) ;
\draw[line width=2pt, style=loosely dashed] (7)--(8) ;
\draw[line width=2pt, style=loosely dashed] (1)--(5) ;
\draw[line width=2pt, style=loosely dashed] (1)--(3) ;
\draw[line width=2pt, style=loosely dashed] (2)--(8) ;
\draw[line width=2pt, style=loosely dashed] (3)--(5) ;
\draw[line width=2pt, style=loosely dashed] (6)--(8) ;
\draw[line width=2pt, style=loosely dashed] (5)--(7) ;
\draw[line width=2pt] (1)--(4) ;
\draw[line width=2pt] (1)--(6) ;
\draw[line width=2pt] (1)--(7) ;
\draw[line width=2pt] (1)--(8) ;
\draw[line width=2pt] (2)--(3) ;
\draw[line width=2pt] (2)--(4) ;
\draw[line width=2pt] (2)--(5) ;
\draw[line width=2pt] (2)--(6) ;
\draw[line width=2pt] (2)--(7) ;
\draw[line width=2pt] (3)--(4) ;
\draw[line width=2pt] (3)--(6) ;
\draw[line width=2pt] (3)--(7) ;
\draw[line width=2pt] (3)--(8) ;
\draw[line width=2pt] (4)--(5) ;
\draw[line width=2pt] (4)--(6) ;
\draw[line width=2pt] (4)--(7) ;
\draw[line width=2pt] (4)--(8) ;
\draw[line width=2pt] (5)--(6) ;
\draw[line width=2pt] (5)--(8) ;
\draw[line width=2pt] (6)--(7) ;

\end{tikzpicture}\\
\columnbreak
\begin{tikzpicture}[scale=.3]

\node[fill, shape=circle] (1) at (247.5:7) {};
\node[fill, shape=circle] (2) at (292.5:7) {};
\node[fill, shape=circle] (3) at (202.5:7) {};
\node[fill, shape=circle] (4) at (337.5:7) {};
\node[fill, shape=circle] (5) at (157.5:7) {};
\node[fill, shape=circle] (6) at (22.5:7) {};
\node[fill, shape=circle] (7) at (112.5:7) {};
\node[fill, shape=circle] (8) at (67.5:7) {};

\draw[line width=2pt, style=loosely dashed] (1)--(2) ;
\draw[line width=2pt, style=loosely dashed] (1)--(5) ;
\draw[line width=2pt, style=loosely dashed] (1)--(7) ;
\draw[line width=2pt, style=loosely dashed] (2)--(6) ;
\draw[line width=2pt, style=loosely dashed] (2)--(8) ;
\draw[line width=2pt, style=loosely dashed] (5)--(7) ;
\draw[line width=2pt, style=loosely dashed] (6)--(8) ;
\draw[line width=2pt, style=loosely dashed] (7)--(8) ;
\draw[line width=2pt] (1)--(3) ;
\draw[line width=2pt] (1)--(4) ;
\draw[line width=2pt] (1)--(6) ;
\draw[line width=2pt] (1)--(8) ;
\draw[line width=2pt] (2)--(3) ;
\draw[line width=2pt] (2)--(4) ;
\draw[line width=2pt] (2)--(5) ;
\draw[line width=2pt] (2)--(7) ;
\draw[line width=2pt] (3)--(4) ;
\draw[line width=2pt] (3)--(5) ;
\draw[line width=2pt] (3)--(6) ;
\draw[line width=2pt] (3)--(7) ;
\draw[line width=2pt] (3)--(8) ;
\draw[line width=2pt] (4)--(5) ;
\draw[line width=2pt] (4)--(6) ;
\draw[line width=2pt] (4)--(7) ;
\draw[line width=2pt] (4)--(8) ;
\draw[line width=2pt] (5)--(6) ;
\draw[line width=2pt] (5)--(8) ;
\draw[line width=2pt] (6)--(7) ;

\end{tikzpicture}\\
\end{center}
\end{multicols}
\caption{Two switching inequivalent signings of $K_{8}$ with the same negative cycle vector $(28,108,336,848,1440,1248)$.}
\end{figure}
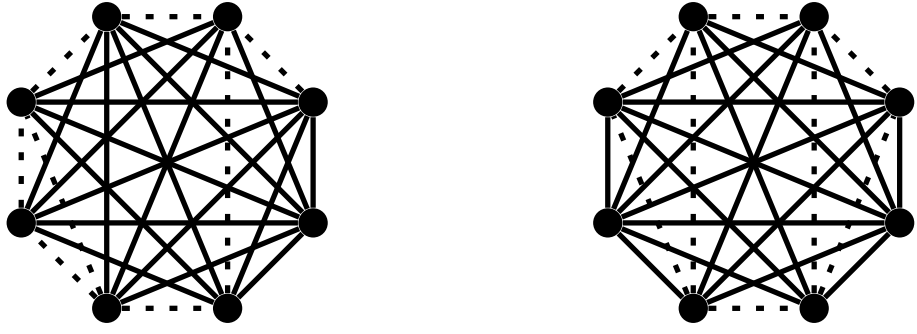

Now we carry out the missing computation of the function $G_{l}$ of Section \ref{pnegm}.  
Consider the signed $K_{n}$'s whose negative edges are $s$ nonadjacent edges, for $0\leq s\leq \lfloor{n}/{2}\rfloor$.
It is straightforward to compute $g_{l}$. For a fixed $k\geq1$ and set $Y$ with $|Y|=k$, we need to form an $l$-cycle using $Y$ and $l-k$ other edges.  (Since $Y$ is a matching, we know that $l\geq2k$.)  
So we choose $l-2k$ of the remaining $n-2k$ vertices, and then create our cycle as follows: imagine contracting the edges in $Y$; the resultant vertices, together with the other $l-2k$ vertices, will form an $l-k$-cycle in the contracted graph (which will eventually give an $l$-cycle in $K_n$). 
Cyclically order these $l-k$ ``vertices''; this orders the vertices in our actual cycle while ensuring the edges from $Y$ remain. There are $\frac{(l-k-1)!}{2}$ ways to do this. Then, we expand the contracted edges to regain them; there are 2 ways to do this for each edge. So we have 
\[g_{l}(Y)=\binom{n-2k}{l-2k}(l-k-1)!\cdot2^{k-1},\] 
whence
\[G_{l}(k)=\binom{n-2k}{l-2k}(l-k-1)!\cdot2^{k-1}.\] 

By Equation \eqref{E:Ms}, $c^{-}_{l}(s)$ is a polynomial in $s$ of degree $d_l=\lfloor l/2 \rfloor$ and the general formula is
\begin{align*}
c^{-}_{l}(s) 
& =\sum_{k=1}^{n}\binom{s}{k} (-4)^{k-1}\binom{n-2k}{l-2k}(l-k-1)!, 
\end{align*}
For example, $c^{-}_{3}(s)=s(n-2)$ and $c^{-}_{4}(s)=s(n^{2}+5n+8)-2s^{2}$.  
This formula for $c^{-}_{l}(s)$ demonstrates that the degrees $d_l$ of the odd polynomials are all distinct, and the same for the even polynomials; consequently our main theorem \ref{T:main} itself implies that the matrix of negative cycle vectors $c^-(s)$ has full rank $n-2$.

\ssubsection{Complete Bipartite Graphs}

We move along to $K_{p,q}$, which always has $p\leq q$.  We use a maximum matching $M_p$, i.e., we set $m=p$.

To get $c^-_{2l}(K_{p,q})$ we compute $g_{2l}$ (where the subscript is now $2l$ because all cycles have even length).  Call the two independent vertex sets $A=\{a_1,\ldots,a_p\}$ and $B=\{b_1,\ldots,b_q\}$. 
For a fixed $k$-edge set $Y=\{ a_{i_1}b_{j_1}, \ldots, a_{i_k}b_{j_k} \} \subseteq M_p$, where $k\leq l$, we need to form a $2l$-cycle using $Y$ and $2l-2k$ other vertices.  Fix one edge $y_1 \in Y$, say $y_1=a_{i_1}b_{j_1}$.  
Choose $l-k$ of the remaining $p-k$ vertices from $A$, in order, in one of $(p-k)_{l-k}$ ways; $l-k$ of the remaining $q-k$ vertices from $B$, also in order, in one of $(q-k)_{l-k}$ ways; and shuffle the sequences together as $(a_{i_{k+1}}, b_{j_{k+1}}, \ldots, a_{i_{l}}, b_{j_{l}})$. 
Insert $Y$ into this $2(l-k)$-sequence by inserting $y_1$ before $a_{i_{k+1}}$ (which we may do because each $Y$ edge must be between an $A$ vertex and a $B$ vertex), treating the resulting sequence as cyclically ordered (which can be done in only one way since the $A$ neighbor of $y_1$ appears after $y_1$); then ordering $Y \setminus \{y_1\}$ in one of $(k-1)!$ ways as $(y_2, \ldots,y_k)$; and finally inserting $y_2,\ldots,y_k$ anywhere into the cycle in one of 
$$\binom{[2(l-k)+1]+[k-1]-1}{[2(l-k)+1]-1}=\binom{2l-k-1}{k-1}$$
ways. 
(When those edges are inserted into the cycle, there is only one way to orient each edge.)  
The net result is that 
\[
G_{2l}(k)=g_{2l}(Y)=(p-k)_{l-k}(q-k)_{l-k} \cdot (k-1)!\binom{2l-k-1}{k-1}.
\] 
Then by Equation \eqref{E:Ms}, for $2\leq l \leq p$,
\begin{align*}
&c_{2l}^-(s)   
= \sum_{k=1}^{p} (s)_k \frac{(-2)^{k-1}}{k} (p-k)_{l-k}(q-k)_{l-k} \binom{2l-k-1}{k-1}.
\end{align*}
This explicit formula for the negative cycle vectors $c^-(s)$, with Theorem \ref{T:main}, implies that $\dim\NCV(K_{p,q})=p=\min(p,q)$.

\ssubsection{The Petersen graph}\label{X:P}

Next we consider the Petersen graph $P$, which has four cycle lengths, 5, 6, 8, and 9, so $\dim\NCV(P) \leq 4$.  
It lacks a permutable $4$-matching.  In fact:  

\begin{theorem}\label{T:cubic}
A $3$-regular graph that is arc transitive cannot have a permutable $4$-matching.
\end{theorem}
  
\begin{proof}
By \cite[Theorem 1.1]{SS2017} an arc-transitive graph with a permutable $m$-matching, where $m\geq4$, must have degree at least $m$.
\end{proof}

The Petersen graph does have a permutable 3-matching, in fact, two kinds.  

The first kind consists of alternate edges of a $C_6$.    
In the language of Theorem \ref{T:main}, we must compute $\mu(l)=|\max\{C_{l}\cap M_{3}\}|$ for each cycle length.  
We find with little difficulty that $\mu(5)=2$, $\mu(6)=3$,  $\mu(8)=2$, and $\mu(9)=3$.  Therefore $|\Delta_{\odd}|=2$ and $|\Delta_{\even}|=2$, whence, despite only having a $3$-matching, we can deduce that $\dim\NCV(P)=4$.
We even know the negative cycle vectors corresponding to negative $0$-, $1$-, $2$-, and $3$-submatchings and the negated signatures; they are (in order of matching size)
\[
\begin{array}{cccc}
(0,0,0,0), &(4,4,8,12), &(6,6,8,10), &(6,10,0,10)\phantom{.} \\
(12,0,0,20), &(8,4,8,8), &(6,8,8,10), &(6,10,0,10).
\end{array}
\]
The bottom vector in each column corresponds to the negated signing.

The second kind of permutable 3-matching consists of three edges at distance 3.  
The first matching type also is three equally spaced edges in a $C_9$, but not every such subset of a $C_9$ is also a set of alternating edges of a $C_6$; the other such subsets are 3-matchings of the second kind.  
This second kind generates negative cycle vectors from negative submatchings and the corresponding negated sign functions whose dimension is only 3, not 4.  (With this matching the negated signatures are switching isomorphic to unnegated signatures.)  This shows that not all permutable $m$-matchings in a graph are equally useful.

\ssubsection{The Heawood graph}\label{X:H}

The Heawood graph $H$ is bipartite and has five cycle lengths, 6, 8, 10, 12, and 14, so $\dim\NCV(H) \leq 5$.  It has a permutable 3-matching, indeed three different kinds, for instance alternate edges of a 6-cycle.  
Using that 3-matching we find that $\mu(6)=3$ (obviously), $\mu(8)=2$, $\mu(10)=3$, $\mu(12)=3$, and $\mu(14)=3$.  
These are two different values, thus $\dim\NCV(H)\geq 2$.  The results for the other two kinds of permutable 3-matching are the same except that $\mu(6)=2$.  In every case $\mu$ has two values.

Our matching method, in principle, cannot prove more because $H$ has no permutable 4-matching (see Theorem \ref{T:cubic}).  Nonetheless we suspect the dimension equals $|\Spec(H)|$.


\ssubsection{Other graphs with permutable perfect matchings, and the cube}

Schaefer and Swartz found all graphs that have a permutable perfect matching.  Besides $K_n$ and $K_{p,p}$ they are the hexagon $C_6$, the octahedron graph $O_3$, and three general examples:  the join $K_p \vee \overline K_p$ of a complete graph with its complement, the \emph{matching join} $K_p \vee_M K_p$ obtained from two copies of $K_p$ by inserting a perfect matching between the two copies, and the matching join $K_p \vee_M \overline K_p$, obtained by hanging a pendant edge from each vertex of $K_p$.

Our treatment of them leads us to one other family, the cyclic prisms $C_p \operatorname{\square} K_2$.

\subsubsection{The simple four}

Trivially, $\dim\NCV(C_6)=1=|\Spec(C_6)|$.

It is easy to verify by hand that $O_3$ satisfies the conditions of Corollary \ref{C:pan}, so $\dim\NCV(O_3)=|\Spec(O_3)|=4$.

As for $K_p \vee_M \overline K_p$, since the pendant edges contribute nothing to cycles,  
\[
\Spec(K_p \vee_M \overline K_p)=\Spec(K_p) \text{ and } \NCV(K_p \vee_M \overline K_p)=\NCV(K_p);
\]
thence $\dim\NCV(K_p \vee_M \overline K_p)=|\Spec(K_p \vee_M \overline K_p)|=p$.

It is also easy to show that $K_p \vee \overline K_p$ satisfies the conditions of Corollary \ref{C:pan}.  Thus, $\dim\NCV(K_p \vee \overline K_p)=|\Spec(K_p \vee \overline K_p)|=2p$.

\subsubsection{The matching join $K_p \vee_M K_p$}

This graph is pancyclic, but its permutable matchings are peculiar.  One kind is any matching in a $K_p$.  A maximum matching $M_{\lfloor p/2 \rfloor}$ in $K_p$, for which $\mu(l)=\min(p,\lfloor l/2 \rfloor)$, hence $\dim\NCV(K_p \vee_M K_p) \geq p$ by reasoning similar to that for $K_p$.  The matching $M^{\vee}_p$ that joins the copies of $K_p$ also prevents a permutable matching from having edges in both copies.  The only other permutable matchings are subsets of $M^{\vee}_p$.  This matching only generates $\lfloor p/2 \rfloor$ switching nonisomorphic signatures since negating a subset of $M^{\vee}_p$ switches to negating the complementary subset.  By itself, therefore, choosing our grand matching $M_m$ to be $M^{\vee}_p$ does not give a better lower bound than $p$.  
Nonetheless we feel the dimension is likely to be $n-2=2p-2$.

The smallest case, $K_3 \vee_M K_3$, is the triangular prism.  The cycle count vector is $(c_3,c_4,c_5,c_6)=(2,3,6,3)$.  There are four unbalanced signatures; see Figure \ref{F:3M3}.  
The negative cycle vectors are linearly independent so $\dim\NCV(K_3 \vee_M K_3) = |\Spec(v)|$, in agreement with Conjecture \ref{Conj:S}.

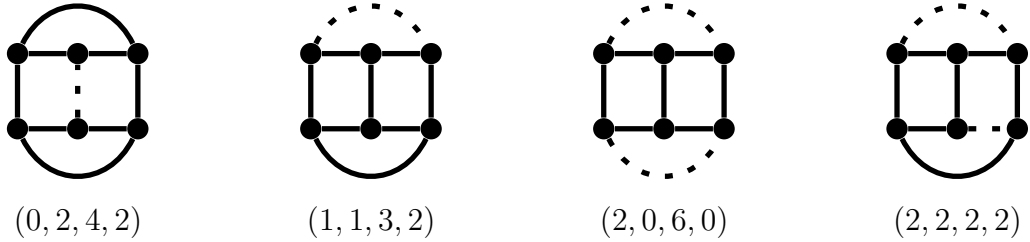
\begin{figure}[ht]
\begin{multicols}{4}
\begin{center}
\begin{tikzpicture}[scale=.2]
\tikzstyle{every node} = [inner sep=.1cm]

\node[fill, shape=circle] (1) at (0,0) {};
\node[fill, shape=circle] (2) at (4,0) {};
\node[fill, shape=circle] (3) at (8,0) {};
\node[fill, shape=circle] (4) at (0,5) {};
\node[fill, shape=circle] (5) at (4,5) {};
\node[fill, shape=circle] (6) at (8,5) {};

\draw[line width=2pt] (1)--(2) ;
\draw[line width=2pt] (2)--(3) ;
\draw[line width=2pt] (1) .. controls(2,-4) and (6,-4) .. (3) ;
\draw[line width=2pt] (4)--(5) ;
\draw[line width=2pt] (5)--(6) ;
\draw[line width=2pt] (4) .. controls(2,9) and (6,9) .. (6) ;
\draw[line width=2pt] (1)--(4) ;
\draw[line width=2pt, style=loosely dashed] (2)--(5) ;
\draw[line width=2pt] (3)--(6) ;

\end{tikzpicture} $(0,2,4,2)$

\columnbreak
\begin{tikzpicture}[scale=.2]
\tikzstyle{every node} = [inner sep=.1cm]

\node[fill, shape=circle] (1) at (0,0) {};
\node[fill, shape=circle] (2) at (4,0) {};
\node[fill, shape=circle] (3) at (8,0) {};
\node[fill, shape=circle] (4) at (0,5) {};
\node[fill, shape=circle] (5) at (4,5) {};
\node[fill, shape=circle] (6) at (8,5) {};

\draw[line width=2pt] (1)--(2) ;
\draw[line width=2pt] (2)--(3) ;
\draw[line width=2pt] (1) .. controls(2,-4) and (6,-4) .. (3) ;
\draw[line width=2pt] (4)--(5) ;
\draw[line width=2pt] (5)--(6) ;
\draw[line width=2pt, style=loosely dashed] (4) .. controls(2,9) and (6,9) .. (6) ;
\draw[line width=2pt] (1)--(4) ;
\draw[line width=2pt] (2)--(5) ;
\draw[line width=2pt] (3)--(6) ;

\end{tikzpicture} $(1,1,3,2)$

\columnbreak
\begin{tikzpicture}[scale=.2]
\tikzstyle{every node} = [inner sep=.1cm]

\node[fill, shape=circle] (1) at (0,0) {};
\node[fill, shape=circle] (2) at (4,0) {};
\node[fill, shape=circle] (3) at (8,0) {};
\node[fill, shape=circle] (4) at (0,5) {};
\node[fill, shape=circle] (5) at (4,5) {};
\node[fill, shape=circle] (6) at (8,5) {};

\draw[line width=2pt] (1)--(2) ;
\draw[line width=2pt] (2)--(3) ;
\draw[line width=2pt, style=loosely dashed] (1) .. controls(2,-4) and (6,-4) .. (3) ;
\draw[line width=2pt] (4)--(5) ;
\draw[line width=2pt] (5)--(6) ;
\draw[line width=2pt, style=loosely dashed] (4) .. controls(2,9) and (6,9) .. (6) ;
\draw[line width=2pt] (1)--(4) ;
\draw[line width=2pt] (2)--(5) ;
\draw[line width=2pt] (3)--(6) ;

\end{tikzpicture} $(2,0,6,0)$

\columnbreak
\begin{tikzpicture}[scale=.2]
\tikzstyle{every node} = [inner sep=.1cm]

\node[fill, shape=circle] (1) at (0,0) {};
\node[fill, shape=circle] (2) at (4,0) {};
\node[fill, shape=circle] (3) at (8,0) {};
\node[fill, shape=circle] (4) at (0,5) {};
\node[fill, shape=circle] (5) at (4,5) {};
\node[fill, shape=circle] (6) at (8,5) {};

\draw[line width=2pt] (1)--(2) ;
\draw[line width=2pt, style=loosely dashed] (2)--(3) ;
\draw[line width=2pt] (1) .. controls(2,-4) and (6,-4) .. (3) ;
\draw[line width=2pt] (4)--(5) ;
\draw[line width=2pt] (5)--(6) ;
\draw[line width=2pt, style=loosely dashed] (4) .. controls(2,9) and (6,9) .. (6) ;
\draw[line width=2pt] (1)--(4) ;
\draw[line width=2pt] (2)--(5) ;
\draw[line width=2pt] (3)--(6) ;

\end{tikzpicture} $(2,2,2,2)$

\end{center}
\end{multicols}
\caption{The four unbalanced switching classes of the prism $K_3 \vee_M K_3$ and their negative cycle vectors.}
\label{F:3M3}
\end{figure}

\subsubsection{Prisms, with cube}

The triangular prism lends support to our belief that $\dim\NCV(K_p \vee_M K_p)=2p-2$.  However, it is atypical since it is also a prism, $C_p \operatorname{\square} K_2$ with $p=3$.  (Prisms with $p>3$ do not have permutable perfect matchings but they make good examples.)  The next prism is the cube, $Q_3=C_4 \operatorname{\square} K_2$.  It is bipartite and has only three cycle lengths: $4$, $6$,  and $8$.  Three unbalanced signatures whose negative cycle vectors are linearly independent are
\begin{enumerate}[\qquad $\sigma_1$,]
\item with one negative edge, $e$.  It has $c^-(\sigma_1)=(2,8,4)$;
\item with a second negative edge, parallel to $e$ and sharing a quadrilateral with it.  It has $c^-(\sigma_2)=(2,12,4)$;
\item with a second negative edge, also parallel to $e$ but not in a common quadrilateral.  It has $c^-(\sigma_3)=(2,4,2)$.
\end{enumerate}
Thus, $\dim\NCV(Q_3)=|\Spec(Q_3)|$, again agreeing with Conjecture \ref{Conj:S}.



\end{document}